\documentclass[12pt]{amsart}
\usepackage{amsmath,amssymb}

\newtheorem{thm}{Theorem}
\newtheorem{lem}[thm]{Lemma}

\newtheorem{cor}[thm]{Corollary}
\theoremstyle{definition}

\newtheorem{rem}[thm]{Remark}
\newcommand{\cH}{{\mathcal H}}
\newcommand{\cK}{{\mathcal K}}
\newcommand{\cL}{{\mathcal L}}
\newcommand{\cM}{{\mathcal M}}
\newcommand{\IB}{{\mathbb B}}

\newcommand{\IF}{{\mathbb F}}

\newcommand{\IN}{{\mathbb N}}

\DeclareMathOperator{\id}{id}
\DeclareMathOperator{\CB}{CB}
\DeclareMathOperator{\cb}{cb}
\newcommand{\OS}{\mathcal O \mathcal S}

\makeatletter
\def\freeprodsize@{\@setfontsize\freeprodsize\@xxpt\@xxpt}
\def\freeprod@{\mathop{\raisebox{-1pt}{\freeprodsize@ $\ast$}}}
\def\freeprod{\freeprod@\displaylimits}
\makeatother

\title{A continuum of $\mathrm{C}^*$-norms on $\IB(H)\otimes \IB(H)$ and related tensor products}
\author{Narutaka Ozawa}
\address{RIMS, Kyoto University, \mbox{606-8502} Japan}
\email{narutaka@kurims.kyoto-u.ac.jp}
\thanks{The first author was partially supported by JSPS (23540233)}
\author{Gilles Pisier}
\address{Texas A{\&}M University, College Station, TX 77843, U.S.A.\hspace*{\fill}\linebreak[4]
and\linebreak[4]\indent
Universit\'e Paris VI, IMJ, Equipe d'Analyse Fonctionnelle, 
Paris 75252, France}
\begin{document}
\begin{abstract}
For any pair $M,N$ of von Neumann algebras such that the algebraic
tensor product $M\otimes N$ admits more than one $\mathrm{C}^*$-norm,
the cardinal of the set of $\mathrm{C}^*$-norms is at least $ {2^{\aleph_0}}$.
Moreover there is
a    family 
  with cardinality $ {2^{\aleph_0}}$ of injective tensor product functors for $\mathrm{C}^*$-algebras in  Kirchberg's sense.
 Let $\IB=\prod_n M_{n}$.
We also show that,
for any non-nuclear  von Neumann algebra $M\subset \IB(\ell_2)$, the set of $\mathrm{C}^*$-norms on $\IB \otimes M$ has 
  cardinality equal to $2^{2^{\aleph_0}}$. 
\end{abstract}
\maketitle

A norm $\alpha$ on an involutive   algebra $A$ is called a $\mathrm{C}^*$-norm if it satisfies 
$$\forall x\in A\quad \alpha(x^*x)=\alpha(x)^2$$
in addition to $\alpha(x^*)=\alpha(x)$ and 
$\alpha(xy)\le \alpha(x)\alpha(y)$ for all $x,y\in A$.
After completion, $(A,\alpha)$ yields a $\mathrm{C}^*$-algebra.
While it is well known that $\mathrm{C}^*$-algebras have a unique $\mathrm{C}^*$-norm,
it is not so for   involutive algebras before completion. For example,  it is well known that
the algebraic tensor product $A\otimes B$ of two $\mathrm{C}^*$-algebras 
may admit distinct $\mathrm{C}^*$-norms, in particular  a minimal one
and a maximal one 
denoted respectively  by $\|\ \|_{\min}$ and $\|\ \|_{\max}$.
When two $\mathrm{C}^*$-norms on   $A\otimes B$ are equivalent, they must coincide
since the completion  has a unique $\mathrm{C}^*$-norm. The $\mathrm{C}^*$-algebras
$A$ such that $\|\ \|_{\min}=\|\ \|_{\max}$ on $A\otimes B$ 
 (or equivalently $A\otimes B$ has a unique $\mathrm{C}^*$-norm) 
 for any other $\mathrm{C}^*$-algebra
$B$ are called nuclear. 
Since they were introduced   in the 1950's, they have been extensively studied in the literature,
notably in the works of Takesaki, Lance, Effros and Lance, Choi and Effros, Connes,
Kirchberg, and many more.
We refer to \cite{Tak} or to \cite{BO} for an account of
these developments.

 In his 1976 paper \cite{Wa}, Simon Wassermann proved that $\IB(H)$ is not nuclear
when $H=\ell_2$ (or any infinite dimensional Hilbert space $H$).
Here $\IB(H)$ denotes the $\mathrm{C}^*$-algebra formed of all the bounded linear operators on   $H$.
 This left open the question whether $\|\ \|_{\min}=\|\ \|_{\max}$ on $\IB(H)\otimes \IB(H)$.
The latter was answered negatively in \cite{JP}. Curiously however the proofs in \cite{JP}
only establish the existence of two inequivalent $\mathrm{C}^*$-norms on $\IB(H)\otimes \IB(H)$,
namely the minimal and maximal ones, leaving open  the likely existence of many more, which is the main result of this note.

It follows from   \cite{JP} that the min and max norms
are not equivalent on     $M\otimes N$ for any pair $M,N$ of von Neumann algebras 
except
if either $M$ or $N$ is nuclear, in which case, of course,  the min and max norms
are equal. In \cite{Wa} Wassermann showed that a von Neumann algebra $M$ is nuclear iff it is 
 ``finite type I of bounded degree". This means that $M$ is (isomorphic to ) a  
 finite direct sum of tensor products of a commutative algebra with a matrix algebra.
 Equivalently, this means that $M$ does not contain the von Neumann algebra $\prod_n M_n$
 as a $\mathrm{C}^*$-subalgebra.

In the first part of this note, we prove that there is at least a continuum of  different 
(and hence inequivalent) $\mathrm{C}^*$-norms
on the algebraic tensor product $\IB(\ell_2)\otimes\IB(\ell_2)$.
 As a corollary, we obtain a continuum of
  injective tensor product functors for $\mathrm{C}^*$-algebras in the sense of Kirchberg
\cite{Ki}.

Let $\IB=\prod_n M_n$. This is the von Neumann algebra the unit ball of which is the product   of the unit balls
of  the matrix algebras $M_n$.
as 
The assertion that there are at least two distinct $\mathrm{C}^*$-norms on 
$\IB(H)\otimes \IB(H)$ (or on  $M\otimes N$ with $M,N$  not nuclear) reduces to the same assertion on $\IB\otimes\IB$, and this is used in \cite{JP}. 
It  turns out to be
 immediate to deduce from \cite{JP} (see Lemma \ref{tak2})
that the cardinality of the set of $\mathrm{C}^*$-norms on $\IB\otimes \IB$ is $\ge c$. 
Unfortunately, however, we do not see how to pass from
$\IB\otimes \IB$ to $\IB(H)\otimes \IB(H)$  in case of more than two $\mathrm{C}^*$-norms.
In any case 
 we will show in \S 2 that the cardinality
of the set of $\mathrm{C}^*$-norms on $\IB\otimes \IB$  (or  $\IB\otimes M$ with $M$ non-nuclear) is $2^{c}$ with $c$ denoting the continuum. 

We end this introduction with some background remarks.

 \begin{rem}\label{R1} It is easy to see that any unital simple $\mathrm{C}^*$-algebra is what algebraists call ``central simple".
  A unital algebra over a field  is called 
  central simple (or centrally simple) if it is simple and  its centre is reduced to the field of scalars. It is classical (see e.g. \cite[p. 151]{Cohn})  that the tensor
  product of two such algebras is again central simple, and a fortiori simple. 
  The kernel of a $\mathrm{C}^*$-seminorm on (the algebraic tensor product) $A\otimes B$ of two $\mathrm{C}^*$-algebras is clearly an ideal.
  Therefore, if $A,B$ are both simple and unital,  
  any $\mathrm{C}^*$-seminorm on (the algebraic tensor product) $A\otimes B$ is a norm
 as soon as it induces a norm on  each of its two factors.
     \end{rem}
      \begin{rem}\label{eas} 
   Let $I$ be a closed ideal in a $\mathrm{C}^*$-algebra $A$. It is well known
   that the maximal $\mathrm{C}^*$-norm is ``projective" in the following sense (see e.g. \cite[p. 92]{BO} or  \cite[p. 237]{P4}):\ for any
   other   $\mathrm{C}^*$-algebra $B$, $I \otimes_{\max} B$ embeds naturally (isometrically)
   in $A  \otimes_{\max} B$
   we have a natural (isometric) identification
   \begin{equation}\label{eq3}(A/I) \otimes_{\max} B= (A  \otimes_{\max} B)/ (I \otimes_{\max} B).\end{equation}
   
   Let $Q(H)=\IB(H)/K(H)$ be the Calkin algebra. By Kirchberg's well known work \cite{Ki} (see \cite[p. 289]{P4}
   or \cite[p. 105]{BO} for more details)
   a $\mathrm{C}^*$-algebra $A$ is exact iff
   \begin{equation}\label{eq4}Q(H)\otimes_{\min} A= (\IB(H)  \otimes_{\min} A)/ (K(H)\otimes_{\min} A).\end{equation}
   Note that $K(H)\otimes_{\min} A=K(H)\otimes_{\max} A$ since $K(H)$ is nuclear. Thus,
   by \eqref{eq3}, if  $A$ is not exact, the minimal and maximal $\mathrm{C}^*$-norms must differ on 
   $Q(H)\otimes A$.
    \end{rem} 
     
    \begin{rem}\label{eas3} Let $A,B,I$ be as in the preceding Remark.
    We can define a $\mathrm{C}^*$-norm on $(A/I)  \otimes  B$ by setting, for any $x\in (A/I)  \otimes  B$,
      \begin{equation}\label{eq00}\alpha(x)=\|x\|_{(A  \otimes_{\min} B)/ (I \otimes_{\min} B)}.\end{equation}
     More precisely, if $y\in A\otimes B$ is any element lifiting $x$ i.e. such that $ (q\otimes Id)(y)=x $
where $q:\ A\to A/I$ denotes the quotient map,
we have
 $$\alpha(x)=\inf\{ \|y+z\|_{\min}\mid z\in  I\otimes_{\min} B  \}.$$
 Since $(I\otimes_{\min} B ) \cap (A\otimes  B )=
 I\otimes B$, this is indeed a norm on $(A/I)\otimes B$.
         
         Let $G\subset B$ be any finite dimensional subspace. Then for any $x\in (A/I)\otimes G$
         we have   \begin{equation}\label{eq01}\alpha(x)=\inf\{ \|y \|_{\min}\mid y\in  A\otimes_{\min} G,  \  (q\otimes Id)(y)=x\}.\end{equation}
         Moreover, the infimum is actually attained. See \cite[\S 2.4]{P4}.

         Now assume that $I$ is   nuclear or merely such that 
    the min and max norms coincide on $I\otimes B$. Then     
     \begin{equation}\label{eq5}(A/I) \otimes_{\min} B= (A/I) \otimes_{\max} B\Rightarrow A  \otimes_{\min} B=A  \otimes_{\max} B.\end{equation}
     More precisely, it suffices   to assume that $\alpha=\|\ \|_{\max}$, i.e. we have 
      \begin{equation}\label{eq6}(A  \otimes_{\min} B)/ (I \otimes_{\min} B)= (A/I) \otimes_{\max} B\Rightarrow A  \otimes_{\min} B=A  \otimes_{\max} B.\end{equation}
          Indeed, this follows from \eqref{eq3} and     $I\otimes_{\min} B=I \otimes_{\max} B$.\\
         
 \end{rem} 
\section{ $\mathrm{C}^*$-norms on  $M \otimes N$  and $\IB(\ell_2) \otimes \IB(\ell_2)$}

We recall the operator space duality which states that $F\otimes_{\min}E^*\subset\CB(E,F)$
isometrically (see Theorem B.13 in \cite{BO} or 
  \cite[p. 40]{P4}).
Namely, for  any operator spaces $E, F$ and any
tensor $z=\sum_k f_k\otimes e^*_k\in F\otimes E^*$, the corresponding map $\varphi_z\colon E\to F$
given by $\varphi_z(x)=\sum_k e^*_k(x)f_k$ satisfies $\|z\|_{\min}=\|\varphi_z\|_{\cb}$.
For a finite dimensional operator space $E$, we denote by $\jmath_E$
the ``identity" element in $E\otimes E^*$.
We note that $\|\jmath_E\|_{\min}=1$ and that $\|\jmath_E\|_\alpha$ is independent of
embeddings $E\hookrightarrow\IB(\ell_2)$ and $E^*\hookrightarrow\IB(\ell_2)$.

For each $d\in\IN$, let $\OS_d$ denote the metric space of all $d$-dimensional operator spaces,
equipped with the cb Banach--Mazur distance.
We recall that by  \cite{JP}  the metric space $\OS_d$ is
non-separable whenever $d\geq3$.
If $A$ is a separable $\mathrm{C}^*$-alge\-bra, then the set $\OS_d(A)$ of all $d$-dimensional
operator subspaces of $A$ is a separable subset of $\OS_d$.

Let $M,N$ be any pair of non-nuclear von Neumann algebras, and let
$\alpha $ be a $\mathrm{C}^*$-norm on $M\otimes N$. Since $\IB$ embeds in both $M$ and $N$, 
any $E\in \OS_d$ admits a completely isometric embedding in both.
We denote by  $\cM_d^\alpha$   the subset of $\OS_d$ that consists of all $E\in\OS_d$
admitting  (completely isometric) realizations $E\subset M$
and $E^*\subset N$ with respect to which
$\|\jmath_E\|_\alpha=1$.\\
For example, one has $\cM_d^{\max}=\OS_d(\mathrm{C}^*\IF_\infty)$ (see \cite{JP}).

\begin{thm}\label{t1} Let $M,N$ be any pair of   von Neumann algebras such that
$M\otimes_{\min} N\not = M\otimes_{\max} N$.
For every $d\in\IN$ and every countable subset $\cL\subset\OS_d$,
there is a  $\mathrm{C}^*$-norm $\alpha$ on $M\otimes N$
such that $\cM_d^\alpha$ is separable and contains $\cL$.
Consequently, there is a  family of   $\mathrm{C}^*$-norms
on $M\otimes N$ with  the cardinality of the continuum.
\end{thm}

\begin{proof}
First note that our assumption ensures that $M,N$ are not nuclear and hence
(by \cite{Wa}) contain a copy of $\IB$.
For each $E\in \cL$, we may assume $E\subset M$ and $E^*\subset N$ completely isometrically.
 Let $A_E\subset M$ be a separable unital $\mathrm{C}^*$-subalgebra containing $E$ completely isometrically.
Let $\IF$ be a large enough free group so that $M$ is a quotient of $C^*(\IF)$.
Consider the $\mathrm{C}^*$-algebraic free product $$A=\freeprod_{E\in \cL} A_E \freeprod C^*(\IF).$$
Let $Q:\ A\to M$ denote the free product of the inclusions $A_E\subset M$ and
the quotient map $C^*(\IF)\to M$, and let $I=\ker(Q)$, so that we have $M\simeq A/I$.
Let $\alpha$ be the $\mathrm{C}^*$-norm defined in \eqref{eq00} with $B=N$.
Using $M\simeq A/I$ we view $\alpha$ as a norm on $M\otimes N$.
 Then for any $E\in \cL$, we have $\alpha(j_E)=1$. Indeed, the inclusion map
 $E\to A_E\to A$ has $cb$ norm 1 and hence defines an element $z\in A\otimes E^*$ 
 with $\|z\|_{\min}=1$   such that $(Q\otimes I) =j_E$.\\
 In the converse direction, let $F\subset M$ be any $d$-dimensional subspace such that,
 viewing $F^*\subset N$ we have
  $\alpha(j_F)=1$. Then, by \eqref{eq01} (applied to $G=F^*$) $j_F$ admits a lifting  $z\in A\otimes F^*$
  with  $\|z\|_{\min}=1$. This yields a completely isometric mapping $F\to A$, showing that $F$
  is completely isometric to a subspace of $A$, equivalently $F\in \OS_d(A)$. But it is easy to check that, for any $d$, the latter set is separable, since any $F\in \OS_d(A)$ is also a subspace of 
  $\freeprod_{E\in \cL} A_E  \freeprod C^*(\IF_\infty)$
  which is  separable (since we assume $\cL$ countable). Thus we have $\cL\subset \cM_d^\alpha$ and $\cM_d^\alpha$ is separable.

 For any $d$-dimensional  $E\subset M$  let $\alpha_E$ be the $\mathrm{C}^*$-norm associated to the singleton $\cL=\{E\}$,
 and let $\mathcal C_E= \cM_d^{\alpha_E}$, so that $E\in \mathcal C_E$.
 Let  $d'(E,F)=\max\{ d_{cb}(E,\mathcal C_F), d_{cb}(F,\mathcal C_E)\}$, where
  $d_{cb}(E,\mathcal C_F)=\inf \{d_{cb}(E,G)\mid  G\in \mathcal C_F\}$.
  By what precedes,  if $d'(E,F)>1$ then necessarily $\alpha_E\not=\alpha_F$ since
 $ \alpha_E(j_F)= \alpha_F(j_E)=1$ implies $d'(E, F)=1$.
  
  By \cite{JP}, for some $\varepsilon>0$, there is     a subset $\mathcal F\subset \OS_d$
 with cardinality $2^{\aleph_0}$ such that
 $d_{cb}(E,F)>1+\varepsilon$ for any $E\not=F\in \mathcal F$.
  Fix  $\xi$ such that $1<\xi<(1+\varepsilon)^{1/2}$.
Since all the $\mathcal  C_E$'s are   separable, we claim that there is a subset
 ${\mathcal F'}\subset {\mathcal F}$ still with cardinality $2^{\aleph_0}$ such that
 $d'(E,F)>\xi $ for any $E\not=F\in \mathcal F'$, and hence the set of $\mathrm{C}^*$-norms
  $\{\alpha_E\mid E\in \mathcal F' \}$ has cardinality $2^{\aleph_0}$.

 Indeed, let $\mathcal F'\subset {\mathcal F}$ be maximal with this property. Then for any $E\in   {\mathcal F}$
there is $F\in   {\mathcal F'}$ such that $d'(E,F)\le \xi$.  Now for any $E$ let $\mathcal  D_E\subset \mathcal  C_E$
be a dense countable subset. Let ${\mathcal F''}=\cup_{E\in {\mathcal F'}} D_E$.
 For any $E\in   {\mathcal F'}$, there is $G=f(E)\in {\mathcal F''}$ such that $d_{cb}(E,G)< (1+\varepsilon)^{1/2}$. 
This defines a function $f:\   {\mathcal F}\to {\mathcal F''}$. 
Assume by contradiction that
 $|\mathcal F'|<|{\mathcal F}|=2^{\aleph_0}$, then also  $|\mathcal F''|<|{\mathcal F}|$,
 and hence the function cannot be injective (``pigeon hole"). Therefore there are $E\not=F\in {\mathcal F}$ such that
 $f(E)=f(F)$ and hence   $d_{cb}(E,F)\le d_{cb}(E,f(E))d_{cb}(F,f(E))<1+\varepsilon$ and we reach a contradiction,
 proving the claim. 
  Thus we obtain a family of $\mathrm{C}^*$-norms $\{\alpha_E\mid E\in \mathcal F'\}$ with cardinality $2^{\aleph_0}$.  \end{proof}

We now turn to admissible norms on $\IB(\ell_2) \otimes \IB(\ell_2)$.
 
We say a $\mathrm{C}^*$-norm $\|\,\cdot\,\|_\alpha$ on $\IB(\ell_2)\otimes \IB(\ell_2)$
is \emph{admissible} if it is invariant under the flip and tensorizes unital completely positive
maps (i.e., for every unital completely positive maps $\varphi\colon \IB(\ell_2)\to \IB(\ell_2)$ the
corresponding map $\varphi\otimes\id$ extends to a completely positive map on
the $\mathrm{C}^*$-algebra $\IB(\ell_2)\otimes_\alpha\IB(\ell_2)$).
Let an admissible $\mathrm{C}^*$-norm $\|\,\cdot\,\|_\alpha$ be given.
We note that for every completely bounded map $\psi$ on $\IB(\ell_2)$ one has
\[
\|\psi\otimes\id\colon\IB(\ell_2)\otimes_\alpha\IB(\ell_2)\to\IB(\ell_2)\otimes_\alpha\IB(\ell_2)\|_{\cb}
=\|\psi\|_{\cb}
\]
(and likewise for  $\id\otimes\psi$),
since $\psi$ can be written as $\|\psi\|_{\cb}S_1^*\varphi(S_1\,\cdot\,S_2^*)S_2$
for some unital completely positive map $\varphi$ on $\IB(\ell_2)$ and isometries
$S_1,S_2$ on $\ell_2$ (see 
Theorem 1.6 in \cite{P4}).

We recall that the density character of a metric space $X$ is the smallest cardinality
of a dense subset. Let $\mathfrak{c}$ denote the cardinality of the continuum.

\begin{lem}
Let $\cH$ be the Hilbert space with density character $\mathfrak{c}$
and consider $\ell_2\subset\cH$. Accordingly, let $\IB(\ell_2)\subset\IB(\cH)$
(non-unital embedding) and $\theta\colon\IB(\cH)\to\IB(\ell_2)$ be the compression.
Then for every unital completely positive map $\varphi\colon\IB(\ell_2)\to\IB(\ell_2)$,
there are a $*$-homo\-mor\-phism $\pi\colon\IB(\cH)\to\IB(\cH)$ and an isometry
$V\in\IB(\ell_2,\cH)$ such that $\varphi(\theta(a))=V^*\pi(a)V$ for every $a\in\IB(\cH)$.
\end{lem}
\begin{proof}
By Stinespring's Dilation Theorem (see   \cite[p. 24]{P4} or  \cite[p. 10]{BO}), 
there are a $*$-repre\-sen\-tation $\pi$ of $\IB(\cH)$ on a Hilbert space $\cK$ 
and an isometry $V\in\IB(\ell_2,\cK)$ such that $\varphi(\theta(a))=V^*\pi(a)V$ 
for every $a\in\IB(\cH)$. We may assume that $\pi(\IB(\cH))V\ell_2$ is dense in $\cK$. 
Since $\varphi(\theta(P_{\ell_2}))=1$, one has $\pi(\IB(\cH))V\ell_2=\pi(\IB(\ell_2,\cH))V\ell_2$.
We claim that the density character of $\IB(\ell_2,\cH)$ is $\mathfrak{c}$. 
Indeed, if we write $\cH=\ell_2(I)$ with $|I|=\mathfrak{c}$, then 
$\IB(\ell_2,\cH)=\bigcup_{J\in [I]^{\IN}}\IB(\ell_2,\ell_2(J))$,
where $[I]^{\IN}$ is the family of countable subsets of $I$.
Since $|[I]^{\IN}|=\mathfrak{c}$ and $\IB(\ell_2)$ has density character $\mathfrak{c}$, 
our claim follows. It follows that $\cK$ has density character $\mathfrak{c}$ and hence 
we may identify $\cK$ with $\cH$. 
\end{proof}
Note that, when $\alpha$ is admissible, $\cM_d^\alpha$ is a closed subset of $\OS_d$.

\begin{thm}\label{t1}
For every $d\in\IN$ and every separable subset $\cL\subset\OS_d$,
there is an admissible $\mathrm{C}^*$-norm $\alpha$ on $\IB(\ell_2)\otimes \IB(\ell_2)$
such that $\cM_d^\alpha$ is separable and contains $\cL$.
Consequently, there is a  family of admissible $\mathrm{C}^*$-norms
on $\IB(\ell_2)\otimes \IB(\ell_2)$ with  the cardinality of the continuum.
\end{thm}
\begin{proof} 
Let $\cL^*=\{ E^* : E\in\cL\}$ and take a separable unital $\mathrm{C}^*$-alge\-bra $A_0$
such that $\OS_d(A_0)$ contains a dense subset of $\cL\cup\cL^*$.
Let $A=\mathrm{C}^*F_\infty\ast\freeprod_{\IN} A_0$ be the unital full free product of
the full free group algebra $\mathrm{C}^*F_\infty$ and countably many copies of $A_0$.
Let $\{ \sigma_i \}$ be the set of all unital $*$-homo\-mor\-phisms from $A$ into $\IB(\cH)$
and $\sigma=\freeprod_i\sigma_i$ be the $*$-homo\-mor\-phism from $\tilde{A}=\freeprod_i A$ to $\IB(\cH)$,
which is surjective. Note that $\OS_d(\tilde{A})=\OS_d(A)$ and hence it is separable.
Denote $J=\ker\sigma$. As in \eqref{eq00}, we induce the $\mathrm{C}^*$-norm $\beta$
on $\IB(\cH)\otimes\IB(\ell_2)$ from $\tilde{A}\otimes_{\min}\IB(\ell_2)$ through $\sigma\otimes\id$,
i.e., for every $z\in \IB(\cH)\otimes\IB(\ell_2)$ one defines
\[
\|z\|_\beta=\inf\{ \|\tilde{z}\|_{\tilde{A}\otimes_{\min}\IB(\ell_2)} : (\sigma\otimes\id)(\tilde{z})=z\}.
\]
Since the infimum s attained,  there is a lift $\tilde{z}\in A\otimes F$ such that $\|\tilde{z}\|_{\min}=\|z\|_\beta$.

Consider $\ell_2\hookrightarrow\cH$ and restrict $\beta$
to $\IB(\ell_2)\otimes\IB(\ell_2)$, which is still denoted by $\beta$.
We claim that for every unital completely positive $\varphi$ on $\IB(\ell_2)$,
the corresponding maps $\varphi\otimes\id$ and $\id\otimes\varphi$ are completely
positive on $\IB(\ell_2)\otimes_\beta\IB(\ell_2)$.
The latter is trivial. For the former, we use the above lemma.
The $*$-homo\-mor\-phism $\pi$ on $\IB(\cH)$ induces a map on $\{\sigma_i\}$ and
thus a $*$-homo\-mor\-phism $\tilde{\pi}$ from $\tilde{A}$ into $\tilde{A}$ such that
$\sigma\circ\tilde{\pi}=\pi\circ\sigma$.
It follows that $\pi\otimes\id$ is a continuous $*$-homomorphism on
$\IB(\cH)\otimes_\beta\IB(\ell_2)$ and hence that $\varphi\otimes\id$ is completely
positive.

We note that $\beta=\min$ on $E\otimes\IB(\ell_2)$ for any $E\in\OS_d(A)$.
Let $E\subset\IB(\ell_2)\subset\IB(\cH)$ and consider the element
$\jmath_E\in E\otimes E^*\subset \IB(\cH)\otimes \IB(\ell_2)$.
If $\|\jmath_E\|_\beta=1$, then $\id_E\colon E\to\IB(\ell_2)$ has a completely contractive
lift into $\tilde{A}$. Indeed, an isometric lifting $\tilde{\jmath}_E \in\tilde{A}\otimes_{\min} E^*$
corresponds to a complete contraction $\theta\colon E\to \tilde{A}$ for which
$\sigma\circ\theta=\id_E\colon E\hookrightarrow\IB(\cH)$. It follows that $\cM_d^\beta\subset\OS_d(A)$.
Finally, take the flip $\beta^{\mathrm{op}}$ of $\beta$ and let $\alpha=\max\{\beta,\beta^{\mathrm{op}}\}$.
\end{proof}

We recall that a \emph{tensor product functor} is a bifunctor $(A,B)\mapsto A\otimes_\alpha B$ which 
assigns in a functorial way a $\mathrm{C}^*$-comple\-tion of each algebraic tensor product 
$A\otimes B$ of $\mathrm{C}^*$-alge\-bras $A$ and $B$. It is said to be \emph{injective} if 
$A_0\hookrightarrow A_1$ and $B_0\hookrightarrow B_1$ gives rise to a faithful embedding 
$A_0\otimes_\alpha B_0\hookrightarrow A_1\otimes_\alpha B_1$. See \cite{Ki}.
For example, the spatial tensor product functor $\min$ is injective, while the maximal one $\max$ is not.

\begin{cor}
There is a     family 
  with cardinality $ {2^{\aleph_0}}$     of  different injective tensor product functors. 
\end{cor}
\begin{proof}
Let $\alpha$ be an admissible $\mathrm{C}^*$-norm $\|\,\cdot\,\|_\alpha$ on $\IB(\ell_2)\otimes\IB(\ell_2)$.
We extend it to a tensor product functor. 
For every finite dimensional operator spaces $E$ and $F$, the norm $\|\,\cdot\,\|_\alpha$ is 
unambiguously defined via embeddings $E\hookrightarrow \IB(\ell_2)$ and $F\hookrightarrow \IB(\ell_2)$.
For every $\mathrm{C}^*$-alge\-bras $A$ and $B$ and $z\in A\otimes B$, we find finite dimensional 
operator subspaces $E$ and $F$ such that $z\in E\otimes F$ and define $\|z\|_\alpha$ to be 
the $\alpha$-norm of $z$ in $E\otimes F$. 
\end{proof}

\section{ $\mathrm{C}^*$-norms on  $\IB \otimes \IB(\ell_2)$ or  $\IB \otimes M$}

\newcommand{\n}{\noindent}
\newcommand{\ms}{\medskip}
\newcommand{\vp}{\varepsilon}
\newcommand{\bb}[1]{\mathbb{#1}}
\newcommand{\cl}[1]{\mathcal{#1}}
\newcommand{\sst}{\scriptstyle}
\newcommand{\ovl}{\overline}
\newcommand{\intl}{\int\limits}

\newtheorem*{rmk}{Remark}



\def\tilde{\widetilde}

\renewcommand{\tilde}{\widetilde}

\newcommand{\spar}{\begin{rotate}{90})\end{rotate}}
\newcommand{\sidpar}{\begin{rotate}{90}\Bigg)\end{rotate}}

\def\RR{\bb R}
\def\CC{\bb C}
\def\KK{\bb K}
\def\E{\bb E}
\def\F{\bb F}
\def\P{\bb P}
\def\T{\bb T}

\def\d{\delta}
\def\NN{\bb N}
\def\RR{\bb R}
\def\PP{\bb P}
\def\CC{\bb C}
\def\KK{\bb K}
\def\a{\alpha}
\def\o{\omega}

Let  $(N(m))$ be any   sequence of positive integers tending to $\infty$ and let
$$B=\prod_m M_{N(m)}.$$
  Actually, the existence
     of a continuum of distinct $\mathrm{C}^*$-norms on $B\otimes B$ can be proved very simply,
     as a consequence of \cite{JP}. 
     \begin{lem}\label{tak2}
     Let $M$ be any $\mathrm{C}^*$-algebra such that $B\otimes_{\min} M\not= B\otimes_{\max} M$. Then there is a continuum of distinct $\mathrm{C}^*$-norms on $B\otimes M$.\end{lem}
      \begin{proof}
     For any infinite subset $s\subset \NN$ we can define a $\mathrm{C}^*$-norm $\gamma_s$
     on $B\otimes M$ by setting
     $$\gamma_s(x)=\max\{ \|x\|_{\min}, \|(q_s\otimes Id)(x)\|_{B_s\otimes_{\max} B}\},$$
     where $B_s=\prod_{m\in s} M_{N(m)}$ and where
     $q_s:\ B\to B_s$ denotes the canonical projection (which is a $*$-homomorphism). 
     Let $\hat B_s =B_s\oplus \{0\}\subset B$ be the corresponding ideal in $B$.
      We claim that if
     $s'\subset \NN$ is another infinite subset such that
     $s\cap s'=\phi$, or merely such that  $t=s\setminus s'$ is infinite,
     then $\gamma_s\not=\gamma_{s'}$. Indeed, otherwise we would find
     that the minimal  and maximal norms coincide on $\hat B_t\otimes M$,  
     and hence (since $B$ embeds in $\hat B_t$ and is
     the range of a unital completely positive projection)
       on $B\otimes M$, contradicting our assumption.      \end{proof}
     By
     \cite{JP} this gives a continuum $(\gamma_s)$ of distinct $\mathrm{C}^*$-norms on $B\otimes B$ or on $B\otimes M$ whenever $M$ is not nuclear.
     Apparently, producing     a family of cardinality $2^{2^{\aleph_0}}$ requires
     a bit more.
       
\begin{thm}\label{t2} There is a  
 family of cardinality $2^{2^{\aleph_0}}$ of mutually distinct (and hence inequivalent)
$\mathrm{C}^*$-norms on $M\otimes B$ for any von Neumann algebra $M $     that is not
nuclear.
\end{thm}
 \begin{rem} Assuming   $M\subset \IB(\ell_2)$ non-nuclear, we note that  
  the cardinality of $\IB(\ell_2)$ and hence of $M\otimes \IB(\ell_2)$ is $c=2^{\aleph_0}$,
  so the set of {\it all} real valued functions of $M\otimes \IB(\ell_2)$ into $\RR$
  has the same cardinal $2^{2^{\aleph_0}}$ as the set of
  $\mathrm{C}^*$-tensor norms.
     \end{rem}

\begin{rem} In the sequel, the complex conjugate $\bar a$ of a matrix $a$ in $M_N$ is 
meant in the usual way, i.e.\ $(\bar a)_{ij} = \ovl{a_{ij}}$.  In general, we 
will need to consider the conjugate $\bar A$ of a $\mathrm{C}^*$-algebra $A$. This is the 
same object but with the complex multiplication changed to $(\lambda,a)\to 
\bar\lambda a$, so that $\bar A$ is anti-isomorphic to $A$.
For any $a\in A$, we denote by $\bar a$ the same element
 viewed as an element of $\bar A$.
 Note that $\bar A$ 
can also be identified with the opposite $\mathrm{C}^*$-algebra $A^{op}$ which is defined 
as the same object but with the product changed to $(a,b)\to ba$. It is easy to 
check that the mapping  $\bar a\to a^*$ is a (linear) $*$-isomorphism from $\bar A$ to 
$A^{op}$.
The distinction between $A$ and $\bar A$ is necessary in general, but not when $A 
= \IB(H)$   since in that case, using $H\simeq \ovl H$,
we have $\ovl{\IB(H)} \simeq  \IB(\ovl H) \simeq
\IB(H)$, and in particular  $\ovl{M_N}  \simeq
M_N $. 
Note however that $H\simeq \ovl H$ depends
on the choice  of a basis so the isomorphism
$\ovl{\IB(H)}\simeq \IB(H)$ is not canonical.
\end{rem}

As in \cite{JP,P5},
our main ingredient will be the fact that the numbers $C(n)$ defined below
are smaller than $n$.
More precisely, it was proved in \cite{HT2} that $C(n)=2\sqrt{n-1}$ for any $n$.
However, it suffices to know for our present purpose that $C(n)<n$ for infinitely many $n$'s
or even merely for some $n$.
This can be proved in several ways for which we refer the reader
to \cite{JP} or  \cite{P4}. See also \cite{P6} for a more recent-somewhat more refined-approach.

For any integer $n\ge 1$, the constant 
$C(n)$ is defined as follows:\ $C(n)$ is the smallest constant $C$ such that for 
each $m\ge 1$, there is $N_m\ge 1$ and an $n$-tuple $[u_1(m), \ldots, u_n(m)]$ 
of unitary $N_m\times N_m$ matrices such that
\begin{equation}\label{eq2}
\sup_{m\ne m'} \left\|\sum^n_{i=1} u_k(m) \otimes\ovl{u_k(m')}\right\|_{\min} 
\le C.
\end{equation}

\n  \label{ga} Throughout the rest of  this note we fix $n>2$ and a constant $C<n$ and we assume given a sequence of 
$n$-tuples $[u_1(m), \ldots, u_n(m)]$ 
of unitary $N_m\times N_m$ matrices satisfying \eqref{eq2}.\\
By compactness (see e.g. \cite{P5}) we may assume (after passing to a subsequence) that
the $n$-tuples $[u_1(m), \ldots, u_n(m)]$ converge in distribution (i.e. in moments)
to an $n$-tuple $[u_1 , \ldots, u_n]$ of unitaries in a von Neumann algebra $M$ equipped
with a faithful normal trace $\tau$. In fact, 
if $\o$ is any ultrafilter refining the selected subsequence, we can take for $M,\tau$ the associated  ultraproduct $M_\o$ of the family
$\{M_{N(m)}\}$ ($m\to \infty$) equipped with normalized traces.

For any subset $s\subset \NN$ and any $1\le k\le n$ we denote by $u_k(s)=\oplus_m u_k(s)(m)$
the element of $B$ defined by $u_k(s)(m)=u_k(m)$ if $m\in s$ and $u_k(m)=0$ otherwise. 
 
Let $\tau_N$ denote the normalized trace on $M_N$.
To any free ultrafilter $\o$ on $\NN$  is associated a tracial state on $B$ defined for any $x=(x_m)\in B$ by
$\varphi_\o(x)=\lim_\o \tau_{N(m)} (x_m).$
The GNS construction applied to that state produces a representation
$$\pi_\o:\ B\to \IB(H_\o).$$
It is classical that $M_\o=\pi_\o(B)$ is a $II_1$-factor and that  $\varphi_\o$
allows to  define a trace $\tau_\o$ on $M_\o$ such that
$ \tau_\o (\pi_\o(b))=\varphi_\o(b)$ for any $ b\in B$.
\begin{rem}
Let $M$ be a finite von Neumann algebra. 
Then for any $n$-tuple $(u_1,\ldots, u_n)$ of unitaries  in $M$ 
\begin{equation}\label{eq0}
\left\|\sum\nolimits^n_1 u_k \otimes \bar u_k\right\|_{M\otimes_{\max} \ovl M} = 
\left\|\sum\nolimits^n_1 u_k\otimes u^*_k\right\|_{M \otimes_{\max} M^{op}} = 
n. \end{equation}
This is a well known fact. See e.g. \cite{BO} or \cite{P4}.
\end{rem}
\begin{lem}\label{lem1} Let $\o\not=\o'$. Consider disjoint subsets $s\subset \NN$ and $s'\subset \NN$
with  $s\in \o$ and $s'\in \o'$, and let 
$$t(s,s')= \sum\nolimits_{k=1}^n u_k(s)\otimes \overline{u_k(s')}\in B \otimes \bar B.$$
Then
 $$\| t(s,s') \|_{B\otimes_{\min} \ovl B}\le C \quad {\rm and}\quad \|[\pi_\o \otimes\ovl{ \pi_{\o'}}]( t(s,s')) \|_{M_{\o}
\otimes_{\max} \overline{M_{\o'} }}
=n.$$
\end{lem}
\begin{proof}    We have obviously
\[
\|t\|_{\min} = \sup_{ {\sst (m,m')\in s\times  s'}} \left\|\sum u_k(m) 
\otimes \ovl{u_k(m')}\right\|
\]
hence $\|t\|_{\min}\le C$.   We now turn to the $\max$ tensor product. We follow \cite{P5}. \\
Let $u_k=\pi_\o(u_k(s))$
and 
$v_k=\pi_{\o'}(u_k(s'))$ so that 
   we 
have
\[
\|[\pi_\o \otimes \ovl{ \pi_{\o'}}]( t(s,s')) \|_{M_{\o}
\otimes_{\max} \overline{M_{\o'} }}
= \left\|\sum u_k\otimes \bar v_k\right\|_{M_{\o}
\otimes_{\max} \overline{M_{\o'} }}.
\]
Now, since we assume that $[u_1(m),\ldots, u_n(m)]$ converges in distribution, 
$(u_1,\ldots, u_n)$ and $(v_1,\ldots, v_n)$ must have the same distribution 
relative respectively to $\tau_{\o}$ and $\tau_{\o'}$. But this implies 
that there is a $*$-isomorphism $\pi$ from the von~Neumann algebra $M(v)\subset M_{\o'}$ 
generated by $(v_1,\ldots, v_n)$ to the one $M(u)\subset M_{\o}$ generated by $(u_1,\ldots, 
u_n)$, defined simply by $\pi(v_k) = u_k$. Moreover, since we are dealing here 
with {\em finite\/} traces, there is a conditional expectation $P$ from $M_{\o'}$ onto $M(v)$.
 Therefore the composition $Q = \pi P$ is a  unital completely 
positive map from $M_{\o'}$ to $M(u)$ such that $Q(v_k) = u_k$. Since such maps
preserve the max tensor products (see e.g. \cite{BO} or \cite{P4}) we have 
$$
\left\|\sum u_k \otimes \bar v_k\right\|_{\max} \ge \left\|\sum   u_k
\otimes \ovl{Q(v_k)} \right\|_{M(u) \otimes_{\max}\ovl{M(u)}}\\
= \left\|\sum u_k \otimes \bar u_k\right\|_{M(u) \otimes_{\max} \ovl{M(u)}}.
$$
But then by \eqref{eq0} we conclude that $\|t(s,s')\|_{\max} = n$.
 \end{proof}
  
For any free ultrafilter ${\o}$  on $\NN$ we denote by $\a_{\o}$
the norm defined on $B\otimes \bar B$ by
$$\forall t\in B\otimes \bar B \qquad \a_{\o}(t)=\max\{\|t\|_{B\otimes_{\min} \bar B}, \|[\pi_\o\otimes Id](t)\|_{M_\o   \otimes_{\max} \ovl{B}}\}.$$
\begin{thm} There is a  family of cardinality $2^{2^{\aleph_0}}$ of mutually distinct (and hence inequivalent)
$\mathrm{C}^*$-norms on $B\otimes \bar B$. More precisely,
the family $\{\a_{\o }\}$ indexed by  free ultrafilters on $\NN$
   is such a family on
$B\otimes \bar B$.
\end{thm}

 \begin{proof} Let $(\o,\o')$   be two distinct   free ultrafilters on $\NN$.
  Let $s \subset \NN$ and $s'\subset \NN$ be disjoint subsets such that
  $s\in \o $ and  $s'\in \o'$. 
  By Lemma \ref{lem1} we have
  $$\a_{ \o} (t(s,s'))\ge \|[\pi_\o \otimes\ovl{ \pi_{\o'}}]( t(s,s')) \|_{M_{\o}
\otimes_{\max} \overline{M_{\o'} }}
 =n$$
  but   since $(\pi_{\o'}\otimes Id)(t(s,s'))=0$
  we have $\a_{ \o'} (t(s,s'))\le C<n$. This shows  $\a_{ \o}$ and $ \a_{ \o'} $
   are different,
  and hence (automatically for $\mathrm{C}^*$-norms) inequivalent.
Lastly, it is well known (see e.g. \cite[p. 146]{CN}) that the cardinality of the set of free ultrafilters on $\NN$ is    
  $2^{2^{\aleph_0}}$.
   \end{proof}   
   
 \begin{proof}[Proof of Theorem \ref{t2}]
 If $M$  is not nuclear, by \cite[Cor.1.9]{Wa} there is an embedding $B\subset M$.
  Moreover, since $B$ is injective, there is a  conditional expectation
 from $M$   to $B$, which guarantees that, for any $A$,
 the max norm on $A\otimes\bar B$ coincides with the restriction of   the max norm on $A\otimes\bar M$. Thus
 we can extend $\alpha_\o$ to a $\mathrm{C}^*$-norm $\tilde \alpha_\o$
   on $B\otimes \bar M$ by setting
$$\forall t \in B\otimes \bar M \qquad \tilde \a_{\o}(t)=\max\{\|t\|_{B\otimes_{\min} \bar M}, \|[\pi_\o\otimes Id](t)\|_{M_\o   \otimes_{\max} \bar{M}}\}.$$
Of course we can replace $M$ by $\bar M$.
 \end{proof}
 \begin{rem} It is easy to see that Theorem \ref{t2} remains valid for any choice of the sequence $(N(m))$
 and in particular it holds if $N(m)=m$ for all $m$, i.e. for  $B=\IB$.
 \end{rem}
 \section{Additional remarks}

  \begin{rem}\label{tak}  Let $G$ be a discrete group such   that its reduced $\mathrm{C}^*$-algebra $A$ is simple.
  We can associate to any  unitary representation $\pi:\ G\to \IB(H_\pi)$
  a $\mathrm{C}^*$-norm $\alpha_\pi$ on $A\otimes A$ as follows.
  Let $\lambda:\ A\to \IB(\ell_2(G)) $ and $\rho:\ A\to \IB(\ell_2(G)) $
  be the left and right regular representations of $G$ linearly extended to $A$.
  This gives us a pair of commuting representations of $A$
on $\ell_2(G)$.  
By the Fell absorption principle (see e.g. \cite[p. 44]{BO} or \cite[p. 149]{P4})
the representation $\pi\otimes \lambda:\ G \to \IB(H_\pi\otimes \ell_2(G))$ is unitarily equivalent to
$I\otimes \lambda$, and hence (since $A$ is assumed simple) it extends to a faithful representation on $A$.
Similarly $I\otimes \rho:\ G \to \IB(H_\pi\otimes \ell_2(G))$ extends to a faithful representation on $A$.
   We define
  $$\forall a,b\in A\times A\quad  \tilde\pi(a\otimes b) =(\pi\otimes \lambda)(a) .(I\otimes \rho)(b) ,$$
  and we denote by $ \tilde\pi$ the canonical extension to $A\otimes A$.
Then for any $x\in A\otimes A$ we set
  $$\alpha_\pi(x)= \| \tilde\pi(x)\|.$$
By Remark \ref{R1}, this is a $\mathrm{C}^*$-\emph{norm} on $A\otimes A$. However,
if we restrict it to the diagonal subalgebra $D\subset A\otimes A$ spanned by $\{\lambda(t)\otimes \lambda(t)\mid t\in G\}$ we  find   for any $x=\sum x(t) \lambda(t)\otimes \lambda(t)$
$$ \| \tilde\pi(x)\|= \|\sum x(t) \pi(t) \otimes \sigma(t)\|$$
where $\sigma(t) \delta_s= \delta_{tst^{-1}}$.

Now, if $G$ is  any non-Abelian free group, $\sigma$
is weakly equivalent to $1\oplus \lambda$ (see \cite{Boz}),
so we have for any such diagonal $x$ (using again $\pi\otimes \lambda\simeq I\otimes \lambda$)
\begin{equation}\label{eq10} \| \tilde\pi(x)\|= \max\{ \|\sum x(t) \pi(t)\| ,  \|\sum x(t) \lambda(t)\|\}. \end{equation}
But it is known (see \cite{Sz}) that there is a continuum of
unitary representations   on a non-Abelian free group   $G$
that are ``intermediate" between $\lambda$ and the universal unitary representation of $G$.
More precisely,  
let $G=\F_k$ be the free group with $k>1$ generators $g_1,\cdots,g_k$.
Let $S_k= \sum_1^k \delta_{g_j} + \delta_{g^{-1}_j}$.
By  \cite[Th. 5]{Sz},
for any number $r\in ((2k-1)^{-1/2}, 1)$,  $G$ admits a unitary representation  
    $\pi_r $ 
  such that
  $$\|\pi_r(S_k)\|= (2k-1)r+ 1/r> 2\sqrt{2k-1}.$$
  By \eqref{eq10}   we have
  $$\|\tilde\pi_r(S_k)\|= (2k-1)r+ 1/r,$$
  and hence  if we define  $x_k=\sum \lambda(g_k) \otimes \lambda(g_k)\in A\otimes A$
  we find
  $$\alpha_{\pi_r} (x_k)=(2k-1)r+ 1/r$$
  which shows that the family of $\mathrm{C}^*$-norms  $\{\alpha_{\pi_r}\mid (2k-1)^{-1/2}< r < 1\}$
  are mutually distinct.
  Thus we obtain in this case a continuum of distinct $\mathrm{C}^*$-norms on $A\otimes A$.\\
  Let $M$ denote the von Neumann algebra generated by $A$ in $\IB(\ell_2(G))$.
  Since $G$ is i.c.c. $M$ is a finite factor and hence (see \cite[p. 349]{Tak})
  is a simple $\mathrm{C}^*$-algebra, thus again automatically central simple. The representation
  $\tilde\pi$ clearly extends to a $*$-homomorphism on $M\otimes M$ which
  is isometric when restricted either to $M\otimes 1$ or $1\otimes M$. 
  Thus we also obtain a  continuum of distinct $\mathrm{C}^*$-norms on $M\otimes M$,
  extending the preceding ones on $A\otimes A$.



     \end{rem} 
  \begin{rem}\label{was}   Let $I\subset A$ and $J\subset B$ be (closed two-sided)
  ideals in two arbitrary $\mathrm{C}^*$-algebras $A,B$.
  Assume that there is only one $\mathrm{C}^*$-norm both on $I\otimes B$ and on
  $A\otimes J$. Let $K=I\otimes_{\min} B+A\otimes_{\min}  J$.
 Then for any pair $\alpha,\beta$ of distinct $\mathrm{C}^*$-norms on $A\otimes B$,
 the quotient spaces $(A\otimes_\alpha B)/K$ must be different
 (note that $I\otimes_{\min} B$, $  A\otimes_{\min}  J  $ and hence  also $K$ are closed in both $A\otimes_\alpha B$ and $A\otimes_\beta B$).
 Therefore the  $\mathrm{C}^*$-norms naturally induced on $(A/I)\otimes (B/J)$
 are also distinct.
  
  For instance, for the Calkin algebra $Q(H)$, we deduce that there are at least $2^{\aleph_0}$ $\mathrm{C}^*$-norms on $Q(H)\otimes \IB(H)$ or
  on $Q(H)\otimes Q(H)$.
  \end{rem}

\subsection*{Acknowledgments} This research was carried out when the first author was staying
at Institut Henri Poincar\'e for the program ``Random Walks and
Asymptotic Geometry of Groups" in 2014. He gratefully
acknowledges the hospitality provided by IHP.
The second author is grateful to  Simon Wassermann for stimulating exchanges.

\end{document}